\documentclass[12pt]{article}

 \usepackage{graphics}
 \usepackage{graphicx}
 \usepackage{epstopdf}
 \usepackage{epsfig}
 \usepackage[all]{xy}
 \usepackage{cite}

 \usepackage{amssymb}
 \usepackage{amsthm,amsmath}

 \usepackage{latexsym, epsfig, psfrag,eepic,colordvi,bm}
 \usepackage{graphics,color}
 \usepackage{amsmath,amsfonts,amssymb,amscd}
 \usepackage[all]{xy}
 \usepackage{epstopdf}
 \usepackage{mathrsfs}

 \usepackage{amsthm,amsmath,amssymb}

 \usepackage{graphicx}
\usepackage{amsthm,amsmath,amssymb}

\usepackage{graphicx,epsfig}

\usepackage[colorlinks=true,citecolor=black,linkcolor=black,urlcolor=blue]{hyperref}
\usepackage{arydshln}

\newcommand{\genstirlingI}[3]{%
	\genfrac{[}{]}{0pt}{#1}{#2}{#3}%
}

\newcommand{\stirlingI}[2]{\genstirlingI{}{#1}{#2}}

\theoremstyle{plain}
\newtheorem{theorem}{Theorem}[section]
\newtheorem{lemma}[theorem]{Lemma}
\newtheorem{corollary}[theorem]{Corollary}
\newtheorem{proposition}[theorem]{Proposition}

\theoremstyle{definition}

\newtheorem{conjecture}[theorem]{Conjecture}

\theoremstyle{remark}

\date{}
\title{\bf Genus distribution polynomials for bicellular bicolored maps all with
real zeros}

\author{Zi-Wei Bai, Ricky X. F. Chen~\footnote{Corresponding author (ORCID: 0000-0003-1061-3049)}\\
	\small School of Mathematics, Hefei University of Technology\\[-0.8ex]
	\small 485 Danxia Road, Hefei, Anhui 230601, P.~R.~China\\[-0.8ex]
	\small\tt 2021111458@mail.hfut.edu.cn, xiaofengchen@hfut.edu.cn
}

\begin{document}

\maketitle
	\tableofcontents
	\newpage
\noindent{\bf Abstract.}
Enumerating bicolored maps and maps according to the numbers (and possibly types)
of edges, faces, white vertices, black vertices and genus has been an important topic arising
in many fields of mathematics and physics. 
In particular, Jackson (1987), Zagier (1995) and Stanley (2011) respectively obtained some expressions for the generating polynomial
of the numbers of one-face bicolored maps with given number of edges and white vertex degree distribution while
tracking the number of black vertices. The cases for multiple faces are harder. In this paper,
	we first obtain the number for that of bicolored
maps with two faces, i.e., bicellular, of arbitrary length distribution, and then derive an explicit formula for the corresponding generating polynomial with respect to genus.
We next prove that the generating polynomial essentially has only real zeros and thus the genus distribution is log-concave.

\vskip 10pt



\section{Introduction}

A map in this study is a $2$-cell embedding of a connected graph with loops and multi-edges allowed in an orientable surface~\cite{lan-zvon}.
A bicolored (bipartite) map is a map where the vertices are colored black and white such that every edge is
incident to two vertices of different colors.
If every white vertex is of degree two in a bicolored map, then we may
replace every white vertex and its incident edges with a single edge to obtain
a map (with only black vertices).
Thus, maps may be viewed as special bicolored maps.
Bicolored maps are also known as Grothendieck's dessins d'enfants in studying Riemann surfaces~\cite{Grothendieck,zap,lan-zvon,kaz-zog15}.
These objects lie at the crossroad of a number of research fields, e.g.,
topology, mathematical physics and representation theory,
and we will be mainly interested in their enumerative and combinatorial properties.

An edge-end of an edge is sometimes called a half-edge (or a dart).
A rooted map is a map where one half-edge is distinguished and called the root.
As for rooted bicolored maps, the root is assumed to be a half-edge around a black vertex by convention.
A labeled bicolored map of $n$ edges is a rooted bicolored map with edges respectively having
a label from the label set  $[n] = \{1, 2,\ldots , n\}$ where in particular
the edge containing the root is labeled $1$.

It is well-known (e.g.~\cite{lan-zvon}) that a map is completely determined by the cyclic ordering
of the half-edges around every vertex (i.e.,~rotation system) and a labeled bicolored map of $n$ edges can be encoded into a triple
of permutations in the set $\mathfrak{S}_n$ of permutations on $[n]$.
Specifically, a triple $(\alpha, \beta, \gamma)$ of permutations such that $\gamma=\alpha \beta$ (compose from right to left)
and the group generated by $\alpha$ and $\beta$ acts transitively on $[n]$
determines a labeled bicolored map in the following manner:
the cycles of $\gamma$ encode the faces, the cycles of $\alpha$ encode the white vertices,
and the cycles of $\beta$ encode the black vertices,
and an edge connects the same label in $\alpha$ and $\beta$, and the root is the half-edge of the edge with label $1$ around a
black vertex.
See Figure~\ref{fig:exam} for an example, where the black vertices give the permutation (in cycle notation) $\beta= (1836)(2754)$,
and the white vertices give the permutation $\alpha= (135)(26478)$.
A face is a region of the embedding surface bounded by a cyclic sequence of edges $(\ldots, i, i', \ldots)$,
where $i'$ is the edge obtained by starting with the edge $i$, going counterclockwisely to the
neighboring edge $e$ around a black vertex, and then going counterclockwisely to the neighboring edge
around a white vertex.
It is not hard to see that a face corresponds to a cycle in $\gamma=(12857)(346)$.
The requirement of transitive action forces the underlying graph to be connected.
If the requirement is not satisfied, the triple may be viewed as a disconnected bicolored map. 

In the following, bicolored maps are always rooted and labeled, but possibly disconnected, unless explicitly stated otherwise.
Additionally, two bicolored maps are viewed the same if the induced triples of permutations are the same.

\begin{figure}
	\centering
	
	\includegraphics[width=0.6\linewidth]{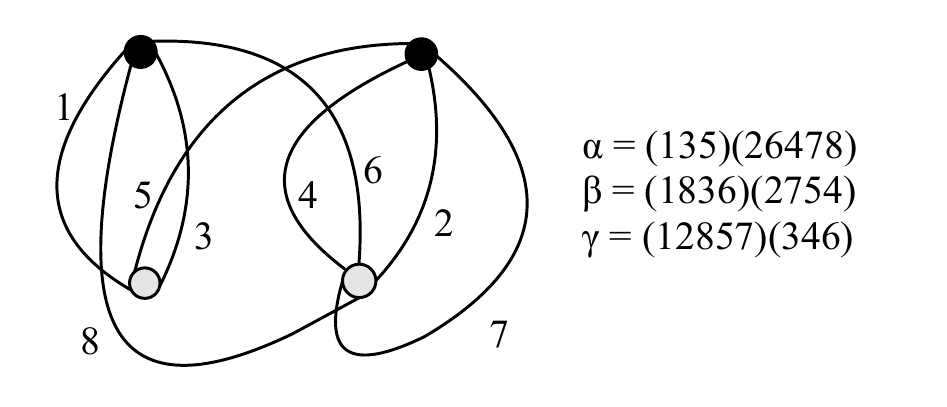}
	\caption{A labeled bicolored map of $8$ edges and genus $2$. }
	\label{fig:exam}
\end{figure}%

According to the Euler characteristic formula, the genus $g$ of a connected bicolored map (defined as the genus of the underlying embedding surface) satisfies:
\begin{align}
	v-e+f=2-2g,
\end{align}
where $v, e, f$ are respectively the numbers of vertices, edges and faces of the map.
This translates to the following relation:
\begin{align}\label{eq:genus}
	\kappa(\alpha)+\kappa(\beta)-n+\kappa(\gamma)=2-2g,
\end{align}
where $\kappa(\pi)$ denotes the number of cycles in the permutation $\pi$.

Since the work of Tutte in 1960s (e.g.,~\cite{tutte63}), the enumeration of bicolored maps (thus maps) has been being centered around studying the enumerative
and combinatorial properties of sets of bicolored maps filtered by different combinations
of these parameters, i.e., the number of edges, the genus, the number of black and/or white vertices (and vertex degree distribution),
and the number of faces (and face length distribution). Many of these studies are based on the permutation triple representation.

Up to our knowledge, if the set for enumeration contains bicolored maps with multiple faces, then usually
the bicolored maps in the set have the same number of edges and genus, while both the numbers of faces and vertices may vary.
See for instance~\cite{tutte63,CV81,BC86,BC91,car-cha15,gao93} counting
maps of $n$ edges and genus $g$.
If the set contains bicolored maps with the same number of faces (and possibly face length distribution), the same number of edges and genus other than zero, then it seems much harder to obtain the size of the set, even asymptotically.
A few results exist (e.g., explicit formulas for the genus zero case~\cite{BMS00}, certain asymptotics in~\cite{BCR93}, recurrences in~\cite{car-cha15,louf}, and functional relations in~\cite{kaz-zog15,chapuy-fang}) and almost no explicit formulas are known for the case of an arbitrary given genus
unless the number of faces is one, i.e., one-face maps or unicellular maps.
In fact, there are vast literature on studying the ``easier" one-face (equivalently, one white vertex or one black vertex) bicolored maps.
See~\cite{zag,Stanely-two,chen-reidys,ber-morales,cff,chr-versatile,Goupil,Jackson,MV,sch-vass,walsh1,harer-zagier,chapuy11} and the references therein.
In particular, formulas for the genus distribution polynomial for one-face
bicolored maps with any prescribed white vertex degree distribution have been obtain in Stanley~\cite{Stanely-two}, and somewhat more implicit in Jackson~\cite{Jackson} and Zagier~\cite{zag}, all relying on the group character approach.
Our main contributions here are explicit formulas for the number of certain bicellular (i.e., two-face) bicolored maps of arbitrary genus and the corresponding genus distribution polynomial, and their consequences.

To state our results more precisely, we introduce some notation first.
The multiset consisting of the
lengths of the disjoint cycles of a permutation $\pi \in \mathfrak{S}_n$ is called the cycle-type of $\pi$, and denoted by $\rho(\pi)$.
We write $\rho(\pi)$ as a partition of $n$, i.e., a nonincreasing positive integer sequence $\lambda=(\lambda_1, \lambda_2, \ldots, \lambda_k)$
such that $\lambda_1+\cdots + \lambda_k = n$.
If in $\lambda$, there are $m_i$ parts equal to $i$, we also write $\lambda$ as $[1^{m_1}, 2^{m_2}, \ldots, n^{m_n}]$.
We often omit $i^{m_i}$ when $m_i=0$ and simply write $i^1$ as $i$.
The number of parts in $\lambda$ is denoted by $\ell(\lambda)=k =\sum_i m_i$.
These two representations will be used interchangeably whichever is more convenient.
An $n$-cycle on $[n]$ is a permutation of cycle-type $(n)$, and
a permutation on the set $[2n]$ of cycle-type $[2^n]$ is called a fixed-point free involution.

Stanley~\cite{Stanely-two} studied the following cycle distribution polynomial:
\begin{align}
	P_{\mu}(x)= \frac{1}{|\mathcal{C}_{\mu}|} \sum_{\alpha \in \mathfrak{S}_n, \, \rho(\alpha)= \mu} x^{\kappa(\alpha \gamma)},
\end{align}
where $\gamma$ is a fixed $n$-cycle and $\mathcal{C}_{\mu}$ is the set of permutations on $[n]$ with cycle-type $\mu$.
This can be interpreted as the distribution of the number of black vertices (thus genus) over all
one-face bicolored maps with $n$ edges and white vertex degree distribution given by $\mu \vdash n$.
Note that the celebrated Harer-Zagier formula~\cite{harer-zagier} is for the case $\mu=[2^n] \vdash 2n$.
In~\cite{Stanely-two}, an expression for $P_{\mu}(x)$ was given in terms of the backward shift operator,
and it was shown that $P_{\mu}(x)$ has purely imaginary zeros.
In order for studying bicolored maps with multiple faces, one may study the above polynomial
with $\gamma$ being a permutation with multiple cycles.
As discussed above, no explicit formulas have been obtained for these polynomials.

Let $\gamma$ be a fixed permutation of cycle-type $[p,n-p]$,
and we study the polynomial 
$$
P_{[p,n-p], \mu}^{\bullet} (x)=\frac{ 1}{|\mathcal{C}_{\mu}|} \sum_{\alpha \in \mathfrak{S}_n, \, \rho(\alpha)= \mu, \atop \mbox{\tiny $<\alpha, \gamma>$ transitive}}   x^{\kappa(\alpha \gamma)}
=\sum_{g \geq 0} \frac{ B^{g}_{p,n}(\mu)}{|\mathcal{C}_{\mu}|\cdot |\mathcal{C}_{[p,n-p]}|}  x^{n-\ell(\mu)-2g},
$$
where $ B^{g}_{p,n}(\mu)$ is number of genus $g$ connected bicellular bicolored maps of face-type
$[p, n-p]$ (i.e., one face of length $p$ and one face of length $n-p$) and white vertex degree distribution (abbr.~WDD) $\mu$, equivalently, $ B^{g}_{p,n}(\mu)$ equals the number of genus $g$ connected bicolored maps of black vertex degree distribution
$[p, n-p]$ and white vertex degree distribution $\mu$.
Sometimes we say bicellular bicolored maps to refer to the latter, as the latter is somehow easier to visualize.
We remark that there exist explicit formulas for maps with two faces~\cite{jackson94,GS-2-face10}.
However, both are only for $\mu=[2^n]$ in our notation.
Our approach here does not easily lead to explicit formulas for $\mu=[2^n]$
but many other $\mu$'s instead. 

Our first results are a formula for $ B^{g}_{p,n}(\mu)$ for any $\mu$ satisfying certain condition and its consequences,
in particular, a polynomial structure of $ B^{g}_{p,n}(\mu)$ and some asymptotics.
We then derive a nice expression for the polynomial itself.
Based on these, we next prove that every root of the polynomial $P_{[p,n-p], \mu}^{\bullet} (x)$ is purely imaginary.
This immediately implies the log-concavity of the genus distribution sequence $B^{g}_{p,n}(\mu), g=0, 1,\ldots$ which is closely related to the well-known, three-decade old
log-concavity conjecture of Gross, Robbins and Tucker~\cite{GRT} for topological embeddings of graphs.
In~\cite{GRT}, the authors proved that the conjecture is true for one-vertex graphs (i.e., bouquets of circles)
based on the enumerative formula of one-face maps in~\cite{Jackson}.
This also partially addresses a problem of Stanley (private communication) that if his~\cite{Stanely-two} counterexample
$P_{(3,3,2), (3,3,2)} (x)$ is the ``smallest'' one that has roots with non-zero real parts.
Some of these results generalize those for $p=2$ in~\cite{bai-chen1} with much more delicate (sometimes quite different) arguments.

The paper is organized as follows.
In Section~\ref{sec2}, we review the main tool that will be
used.
In Section~\ref{sec3}, we compute $B^g_{p,n}(\mu)$ and $P^{\bullet}_{[p,n-p], \mu}(x)$ for any $\mu$ having the smallest part larger than $p$ (i.e., the minimum degree of the white vertices is larger than $p$). The obtained formulas are presented in Theorem~\ref{thm:main-1}.
As a consequence, we obtain the exact average genus of the corresponding bicolored maps, an analogue of the Harer-Zagier formula~\cite{harer-zagier} and
particularly a polynomial structure on the numbers of bicolored maps.
In Section~\ref{sec4}, via some novel arguments, we prove that every root of $P^{\bullet}_{[p,n-p], \mu}(x)$
has a real part zero (Theorem~\ref{thm:zero-root}) which implies the log-concavity of the genus distribution sequence of  $B^g_{p,n}(\mu), \, g=0,1,2,\ldots$

\section{Group characters} \label{sec2}

Let $\mathcal{C}_{\lambda}$ denote the conjugacy class of $\mathfrak{S}_n$ containing
permutations of the same cycle-type $\lambda$.
For $\lambda=[1^{m_1}, 2^{m_2},\ldots, n^{m_n}] \vdash n$,
the number of elements contained in $\mathcal{C}_{\lambda}$ is well known to be
$$
|\mathcal{C}_{\lambda}|=\frac{n!}{z_{\lambda}}, \quad \mbox{where $z_{\lambda}= \prod_{i=1}^n i^{m_i} m_i!$.}
$$
The permutations on $[n]$ having exactly $k$ cycles are counted by the signless Stirling
number of the first kind $\stirlingI{n}{k}$ which satisfies
	\begin{align}
	(x)_n:=x(x-1)\cdots (x-n+1) = \sum_{k=1}^n (-1)^{n-k}\stirlingI{n}{k} x^k .
\end{align}

Let $\chi^{\lambda}$ denote the character associated to the irreducible representation of $\mathfrak{S}_n$ indexed by $\lambda \vdash n$,
and $f^{\lambda}$ denote its dimension.
The readers are invited to consult Stanley~\cite{stan-ec2}, and Sagan~\cite{sagan} for the character theory of symmetric groups.

\begin{lemma}[The hook length formula~\cite{stan-ec2,sagan}]
	Suppose $\lambda \vdash n$. Then,
	\begin{align}
		f^{\lambda} = \frac{n!}{ \prod_{u \in \lambda} h(u)},
	\end{align}
where for $u=(i,j)\in \lambda$, i.e., the $(i,j)$ cell in the Young diagram of $\lambda$, $h(u)$ is the hook length of
the cell $u$.
\end{lemma}

 Suppose $C$ is a conjugacy class of $\mathfrak{S}_n$ indexed by $\alpha \vdash n$.
We view $\chi^{\lambda}(\alpha)$ and $\chi^{\lambda}(C)$ as the same.
Let
\begin{align*}
	\mathfrak{m}_{\lambda,m}= \prod_{u\in \lambda} \frac{m+c(u)}{h(u)}, \qquad 	\mathfrak{c}_{\lambda,m} &= \sum_{i=0}^m (-1)^i {m \choose i} \mathfrak{m}_{\lambda,m-i}.
\end{align*}
where for $u=(i,j)\in \lambda$, $c(u)=j-i$.
The following theorem was proved in~\cite{chr-hur}.

\begin{theorem}[Chen~\cite{chr-hur}] \label{thm:main-recur11}
	Let $\xi_{n,m}(C_1,\ldots, C_t)$ be the number of tuples $( \sigma_1,\sigma_2,\ldots ,\sigma_t)$
	such that the permutation $\pi=  \sigma_1\sigma_2\cdots \sigma_t$ has $m$ cycles, where $\sigma_i $
	belongs to the conjugacy class ${C}_i$ of $\mathfrak{S}_n$. Then we have
	\begin{align}\label{eq:main-recur11}
		\xi_{n,m}(C_1,\ldots, C_t)
		= \bigg(\prod_{i=1}^{t}|C_i| \bigg)\sum_{k = 0}^{n-m}  \stirlingI{m+k}{m}  \frac{ (-1)^k }{(m+k)!}  W_{n,m+k}(C_1,\ldots, C_t),
	\end{align}
	where for any $r>0$
	\begin{align}\label{eq:W}
		W_{n,r}(C_1,\ldots, C_t)=  \sum_{\lambda \vdash n} \frac{\mathfrak{c}_{\lambda,r} } {(f^{\lambda})^{t-1}}	  \prod_{i=1}^t \chi^{\lambda}(C_i).
	\end{align}
	
\end{theorem}

There exists a formula for the multivariate generating function of the numbers $\xi_{n,m}$ in
Louf~\cite[Theorem 2.1.4]{louf-phd}. But it is not obvious to us how Theorem~\ref{thm:main-recur11} may be easily extracted from
the generating function.
The case $m=n$ of eq.~\eqref{eq:main-recur11}
reduces to the famous Frobenius identity (e.g.~\cite{lan-zvon}). See Chen~\cite{chr-hur} for discussion.
The second author~\cite{chen} has applied Theorem~\ref{thm:main-recur11} to studying one-face bicolored maps,
and a plethora of results, e.g., the Harer-Zagier formula~\cite{harer-zagier} and the Jackson formula~\cite{Jackson-combinatorial,sch-vass}, have been obtained in a unified way.

\section{Enumeration of bicellular bicolored maps} \label{sec3}

Recall WDD stands for white vertex degree distribution.
In this paper, as for $P_{[p,n-p], \mu}^{\bullet}(x)$, we only study the following two cases for WDD $\mu$ of these bicolored maps:
(i) any $\mu$ with its minimum size $\min(\mu)$ of a part no less than $p+1$, and (ii)
$\mu=[p, n-p]$.
The computation for the former is presented in this section in detail,
while the latter will be separately considered in Section~\ref{sec5}.
Note that when $p<< n-p$, the considered two cases here already contribute
a significant portion of all possibilities.

Let $P_{[p,n-p], \mu}(x)$ denote the counterpart of $P_{[p,n-p], \mu}^{\bullet}(x)$ that additionally counts the corresponding  disconnected bicellular bicolored maps.
For both mentioned cases, we will first compute $P_{[p,n-p], \mu}(x)$ and then obtain $P_{[p,n-p], \mu}^{\bullet}(x)$ by
subtracting the contribution from the disconnected maps if there are.
By construction, the (possibly) disconnected genus distribution polynomial
\begin{align*}
	P_{[p,n-p], \mu}(x) &= \frac{1}{|\mathcal{C}_{\mu}| \cdot |\mathcal{C}_{[p, n-p]}|} \sum_{m>0} \xi_{n,m }(\mathcal{C}_{\mu}, \mathcal{C}_{[p, n-p]}) x^m\\
	&= \sum_{m>0} \sum_{k=0}^{n-m}  \stirlingI{m+k}{m}  \frac{ (-1)^k }{(m+k)!}  W_{n,m+k}(\mathcal{C}_{\mu},\mathcal{C}_{[p, n-p]}) x^m,
\end{align*}
where
the corresponding $W$-number is as follows:
	\begin{align}\label{eq:WW}
	W_{n,r}=  \sum_{\lambda \vdash n} \frac{\mathfrak{c}_{\lambda,r} }{	 f^{\lambda} }  \chi^{\lambda}([p, n-p]) \, \chi^{\lambda}(\mu) .
\end{align}

 Throughout the rest of the present section, we assume
\begin{align*}
	\mu=(\mu_1, \mu_2, \ldots, \mu_d)= [l^{n_l}, \ldots, q^{n_q}] \vdash n,
\end{align*}
where $l \geq p+1$ and $q \geq p+1$ are respectively the minimum size and maximum size of a part in $\mu$.
We also assume $0<p<n-p$ and $n-2p\geq 2$ if not explicitly stated otherwise.

Next, we devote our effort to the computation of the $W$-numbers, where
the evaluation of the irreducible characters involved are based on the Murnaghan-Nakayama rule (see, e.g., Stanley~\cite{stan-ec2}, and Sagan~\cite{sagan}).

\begin{lemma} \label{lem:chi-pnp}
	For $p>0$ and $\lambda \vdash n \geq 2p+1$, we have
	\begin{itemize}
		\item The character
		\begin{align*}
			\chi^{\lambda}([p,n-p])= \begin{cases}
				(-1)^j ,  & \mbox{if }\lambda =[1^j,n-j], \, 0 \leq j \leq p-1, \\
				(-1)^{j+1} , & \mbox{if } \lambda =[1^j,n-j], \mbox{  $n-p \leq j \leq n-1$};
			\end{cases}
		\end{align*}
		\item If $\lambda =[1^j, 2^k, p-k+1, n-j-k-p-1]$ where $0\leq k \leq p-1$ and $0 \leq j \leq n-2p-2$, then $\chi^{\lambda}([p,n-p])= (-1)^{j+1}$;
		\item If $\lambda =[1^j, 2^k, p-k-j, n-k-p]$ where $0\leq k \leq p-2$ and $0 \leq j \leq p-k-2$, then $\chi^{\lambda}([p,n-p])= (-1)^{j}$;
		\item If $\lambda= [1^j, 2^k, n-p-k-j, p-k]$ where $n-2p \leq j \leq n-p-2$ and $0 \leq k \leq n-p-j-2$, then $\chi^{\lambda}([p,n-p])= (-1)^{j}$; 
		\item For any other $\lambda \vdash n$, $\chi^{\lambda}([p,n-p])=0$.
	\end{itemize}
\end{lemma}
\begin{proof}
	According to the Murnaghan-Nakayama rule, the $\lambda$'s for which $\chi^{\lambda}([p,n-p])$ may return a nonzero value
	correspond to the Young diagrams in Figure~\ref{fig:char-p1}.
	\begin{figure}
		\centering
		
		\includegraphics[width=0.8\linewidth]{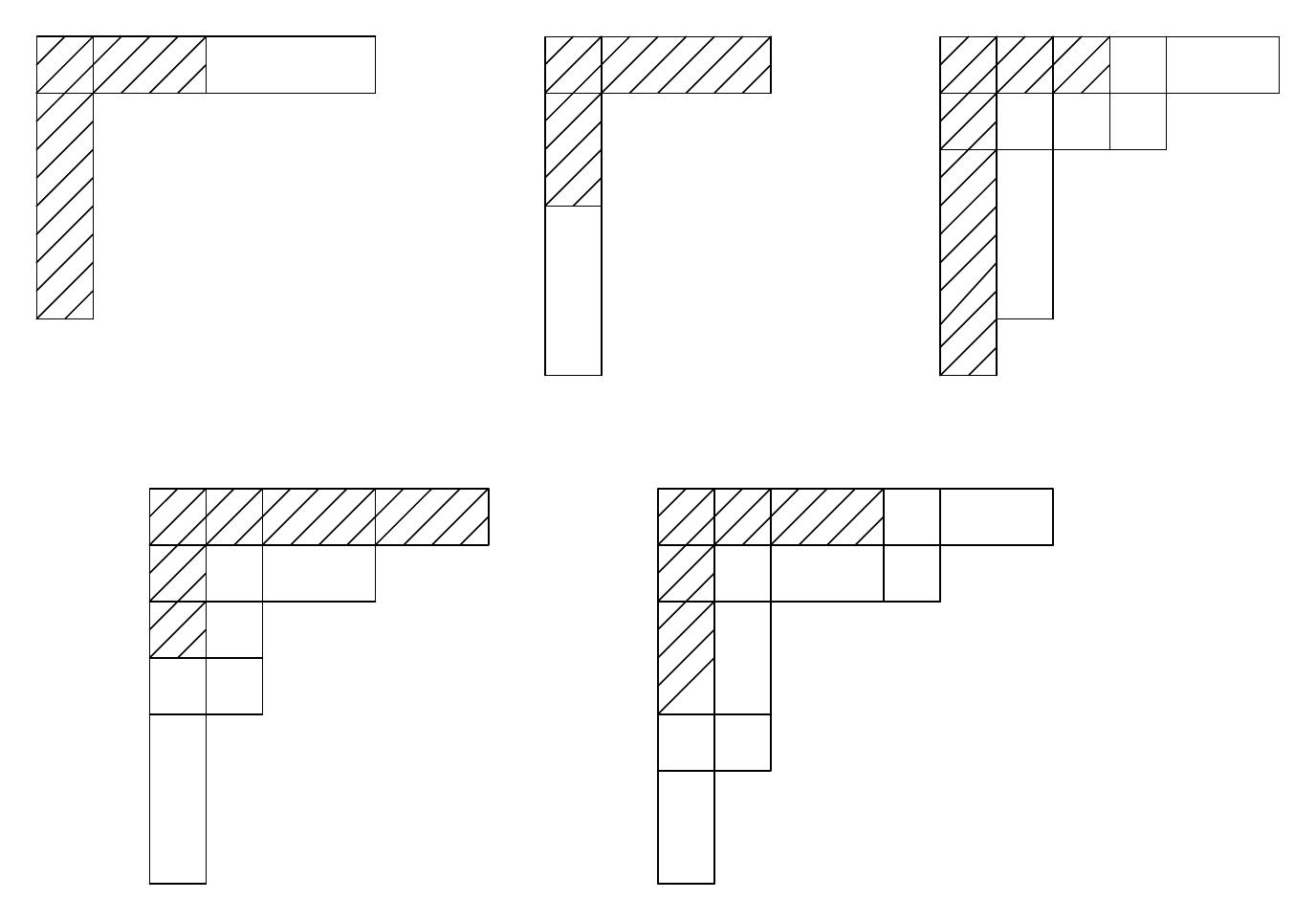}
		\caption{Young diagrams of shape $\lambda$ that $\chi^{\lambda}([p,n-p])$ may not vanish, where the size of the rim hook with slash is $p$. }
		\label{fig:char-p1}
	\end{figure}%
	The rest is straightforward and the lemma follows.
\end{proof}

The calculation of $\chi^{\lambda}(\mu)$ is crucial and generally hard.
However, we observe that only the values for $\lambda$ that give a non-zero value in Lemma~\ref{lem:chi-pnp} are needed.

\begin{lemma}\label{lem:beta}
There holds	
	\begin{align}\label{eq:beta-simple}
		\chi^{\lambda}(\mu)= \begin{cases}
			(-1)^j ,  & \mbox{if } \lambda =[1^j,n-j], \, 0 \leq j \leq p-1, \\
			(-1)^{j+1-\ell(\mu)} , & \mbox{if }\lambda =[1^j,n-j], \mbox{  $n-p \leq j \leq n-1$}.
		\end{cases}
	\end{align}
For $\lambda =[1^j, 2^k, p-k+1, n-j-k-p-1]$ where $0\leq k \leq p-1$ and $0 \leq j \leq n-2p-2$,
	\begin{align}
		\chi^{\lambda}(\mu) &= 	\sum_{{j_l}=0}^{n_{l}-1} \sum_{{j_{l+1}}=0}^{n_{l+1}} \cdots \sum_{{j_q}=0}^{n_q} {{n_{l}-1}\choose{j_{l}}}
		{n_{l+1}\choose j_{l+1}}\cdots{{n_q}\choose j_q} \nonumber\\
		& \qquad \qquad \qquad \times \big\{ (-1)^{j-\sum_{i\geq l} j_{i} } \delta_{j+p+1-l,\sum_{i\geq l}ij_i}+ (-1)^{j+1-\sum_{i\geq l} j_{i}} \delta_{j+p+1,\sum_{i\geq l}ij_i} \big\} ,
	\end{align}
	where $\delta_{x, y}=1$ if $x=y$, and $0$ otherwise, and
	$\chi^{\lambda}(\mu)=0$
	for any other $\lambda$ that $\chi^{\lambda}([p, n-p])$ may not be zero.
\end{lemma}

\proof Eq.~\eqref{eq:beta-simple} is easy to verify.
We only prove the case $\lambda =[1^j, 2^k, p-k+1, n-j-k-p-1]$.
There are only two scenarios that $\chi^{\lambda}(\mu)$ may not be zero for $\min(\mu) \geq p+1$ when applying the Murnaghan-Nakayama rule:
\begin{itemize}
\item[(i)] The rim hook (or border strip) containing the cell $(1,1)$ is of size $l\geq p+1$ and also contains the cells
$(1,2), \ldots, (1,p-k)$ but not the cell $(1, p-k+1)$. See an illustration in Figure~\ref{fig:char-mu} left;
\item[(ii)] The rim hook containing the cell $(1,1)$ is of size $l\geq p+1$ and also contains the cells
$(2,1), \ldots, (k+1,1)$ but not the cell $(k+2, 1)$. See an illustration in Figure~\ref{fig:char-mu} right.
\end{itemize}
\begin{figure}
	\centering
	
	\includegraphics[width=0.9\linewidth]{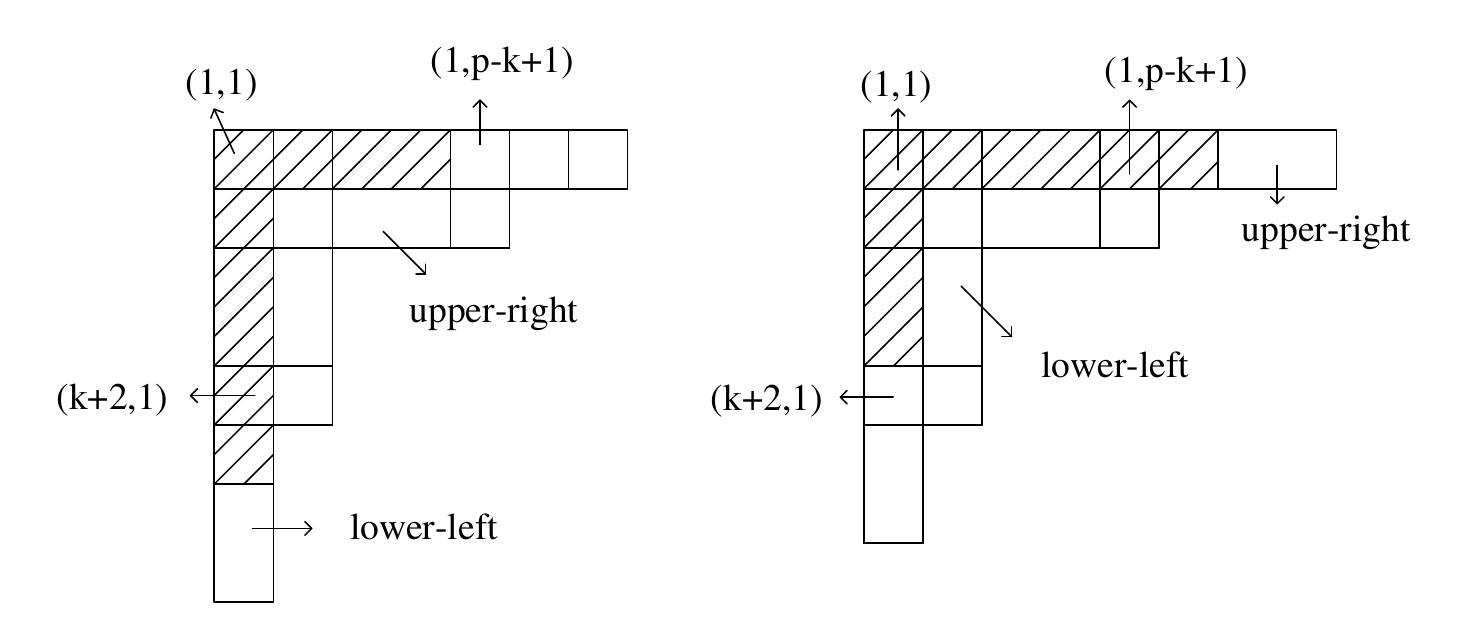}
	\caption{Two scenarios that $\chi^{\lambda}(\mu)$ may not be zero, scenario (i) (left) and scenario (ii) (right),
	where the rim hook with slash is of size $l>p$.}
	\label{fig:char-mu}
\end{figure}%

For (i), the rim hook containing the cell $(1,1)$ necessarily contains the cell $(k+2, 1)$
since the size of the rim hook is at least $p+1$. Then, this rim hook separates the remaining diagram
into two independent pieces: the lower-left piece of size $j+p+1-l$ and the upper-right piece of size 
$n-j-p-1$. Next, we just need to fill each piece with the remaining $n_l-1$ rim hooks of size $l$
and $n_i$ rim hooks of size $i$ for $l < i \leq q$.
Suppose $j_i$ rim hooks of size $i$ go to the lower-left piece
and the rest goes to the upper-right piece.
It is not hard to see that the total height of these rim hooks {\color{blue}(including the one containing $(1,1)$)} is
$$
{j+k+1- \sum_{i \geq l} j_i +k+1}.
$$
Thus, the contribution of such a filling to $\chi^{\lambda}(\mu) $ is given by 
$(-1)^{j- \sum_{i \geq l} j_i} $.
Therefore, the contribution of the first scenario is 
$$
\sum_{{j_l}=0}^{n_{l}-1} \sum_{{j_{l+1}}=0}^{n_{l+1}} \cdots \sum_{{j_q}=0}^{n_q} {{n_{l}-1}\choose{j_{l}}}
{n_{l+1}\choose j_{l+1}}\cdots{{n_q}\choose j_q} 
 \big\{ (-1)^{j-\sum_{i\geq l} j_{i} } \delta_{j+p+1-l,\sum_{i\geq l}ij_i}\big\}.
$$
Analogously, the contribution of the second scenario to $\chi^{\lambda}(\mu)$ is 
$$
\sum_{{j_l}=0}^{n_{l}-1} \sum_{{j_{l+1}}=0}^{n_{l+1}} \cdots \sum_{{j_q}=0}^{n_q} {{n_{l}-1}\choose{j_{l}}}
{n_{l+1}\choose j_{l+1}}\cdots{{n_q}\choose j_q} 
\big\{ (-1)^{j+1-\sum_{i\geq l} j_{i} } \delta_{j+p+1,\sum_{i\geq l}ij_i}\big\},
$$
and the proof follows. \qed

Next, we compute the factor $\frac{\mathfrak{c}_{\lambda,r} }{	 f^{\lambda} }$ in eq.~\eqref{eq:WW}, only for these $\lambda$'s which
return a non-zero value in Lemma~\ref{lem:beta}.
First of all, we have
\begin{align}\label{eq:cf}
\frac{\mathfrak{c}_{\lambda,r} }{	 f^{\lambda} }=\frac{ \sum_{i=0}^r (-1)^i {r \choose i}  \prod_{u\in \lambda} \frac{r-i+c(u)}{h(u)} }{\frac{n!}{ \prod_{u \in \lambda } h(u)}} =  \sum_{i=0}^r  \frac{(-1)^i}{n!}  {r \choose i}  \prod_{u\in \lambda} \{r-i+c(u)\} .
\end{align}
With this, the following two lemmas are obvious and their proofs are omitted.

\begin{lemma} \label{lem:3}
Suppose $\lambda=[1^j,n-j]$ and $r>0$. Then, we have
\[
\frac{\mathfrak{c}_{\lambda, r}}{f^{\lambda }}= \sum_{i=0}^r (-1)^i {r \choose i } {r-i+n-j-1 \choose n}.
\]
\end{lemma}

\begin{lemma}\label{lem:cf}
	For $\lambda =[1^j, 2^k, p-k+1, n-j-k-p-1]$ where $0\leq k \leq p-1$ and $0 \leq j \leq n-2p-2$,
\begin{align*}
		\frac{\mathfrak{c}_{\lambda, r}}{f^{\lambda}}= \frac{p! (n-p)! }{n!} \sum_{i=0}^r (-1)^i
		{r \choose i} {r-i+p-k-1 \choose p} {r-i+n-j-k-p-2 \choose n-p} .
	\end{align*}
\end{lemma}

In the rest of the paper, some further notation will be used:
as usual, the notation $[y^n] f(y)$ means taking the coefficient of the term $y^n$ in the power series expansion of $f(y)$,
and for any $\mu=(\mu_1,\ldots, \mu_d) \vdash n$,
{
\begin{align*}
	\mathcal{V}_{\mu}(y) :=\prod_{i=1}^d \{ (1+y)^{\mu_i} -1\} .
\end{align*}
}

Now we come to our first main theorem.

\begin{theorem}\label{thm:main-1}
	Suppose $0<p \leq n-p$ and $g \geq 0$. Then, we have
	\begin{align} \label{eq:poly-gen}
		\frac{B^g_{p,n}(\mu)}{|\mathcal{C}_{[p, n-p]}| \cdot |\mathcal{C}_{\mu}|}	& = \sum_{k=0}^{2g} \frac{(-1)^k \stirlingI{n-d-2g+k}{n-d-2g}}{(n-d-2g+k)!} W_{n,n-d-2g+k} ,\\
		{n \choose p}	P_{[p, n-p] , \mu}^{\bullet}(x) &=  [y^{n-p}]  \mathcal{V}_{\mu}(y) \sum_{i=0}^{p-1} {x+i \choose p} (1+y)^{x+i-p},
	\end{align}
	where 
	\begin{align*}
		W_{n, r}=  \frac{p! (n-p)! }{n!} [y^{n-p}]  \mathcal{V}_{\mu}(y) \sum_{a \geq r-p} y^a {p+a \choose a} {p \choose p+a -r +1}.
	\end{align*}

\end{theorem}
\proof
At first, it is easy to show that bicolored maps with face-type (equivalently black vertex degree distribution) $[p, n-p]$
and WDD $\mu$ with $\min(\mu) \geq p+1$ are always connected. Consequently, we have 
$$
P_{[p, n-p], \mu}^{\bullet}(x) = P_{[p, n-p], \mu}(x), \quad B_{p,n}^g(\mu) = \xi_{n,n-d-2g}(\mathcal{C}_{[p,n-p]}, \mathcal{C}_{\mu} ).
$$
We next compute $P_{[p, n-p], \mu}(x)$, and
first assume $p+2 \leq n-p$.

In the light of Lemma~\ref{lem:chi-pnp} and Lemma~\ref{lem:beta}, we can break the quantity $W_{n, r}$ into the following form:
\begin{align*}
	{	W_{n, r}} &=  \sum_{\lambda=[1^j, n-j] \atop 0 \leq j \leq p-1} \frac{\mathfrak{c}_{\lambda,r} }{	 f^{\lambda} }  \chi^{\lambda}([p, n-p]) \chi^{\lambda}(\mu) 
	+ \sum_{\lambda=[1^j, n-j], \atop n-p \leq j \leq n-1} \frac{\mathfrak{c}_{\lambda,r} }{	 f^{\lambda} }  \chi^{\lambda}([p, n-p]) \chi^{\lambda}(\mu) \\
	& \quad +\sum_{\lambda=[1^j, 2^k, p-k+1, n-j-k-p-1], \atop 0\leq k \leq p-1, \, 0 \leq j \leq n-2p-2} \frac{\mathfrak{c}_{\lambda,r} }{	 f^{\lambda} }  \chi^{\lambda}([p, n-p]) \chi^{\lambda}(\mu).
\end{align*}
We compute the third sum first, and observe that the first and the second sums will be
canceled out with some terms from the third sum. With various manipulation techniques,
we eventually 
arrive at
	\begin{align}\label{eq:w-number}
		W_{n, r} = [y^{n-p}] \frac{p! (n-p)! }{n!} \mathcal{V}_{\mu}(y) \sum_{a \geq r-p} y^a {p+a \choose a} {p \choose p+a -r +1}.
	\end{align}
	The details for obtaining eq.~\eqref{eq:w-number} can be found in the appendix.
	According to Theorem~\ref{thm:main-recur11}, we obtain the formula for $B^g_{p,n}(\mu)$.

Finally, we compute the generating function.
\begin{align*}
	& \quad P_{[p, n-p], \mu}(x) = \sum_{m \geq 1} \sum_{k=0}^{n-m} (-1)^k \frac{\stirlingI{m+k}{m}}{(m+k)!} W_{n, m+k} x^m \\
	&= \sum_{m \geq 1} \sum_{k=0}^{n-m}  \frac{ (-1)^k \stirlingI{m+k}{m}}{(m+k)!} \frac{p! (n-p)! }{n!} [y^{n-p}]  \mathcal{V}_{\mu}(y) \sum_{a \geq m+k -p} {p+a \choose a} y^a {p \choose p+a+1-m-k} x^m \\
	&= \frac{p! (n-p)! }{n!} [y^{n-p}]  \mathcal{V}_{\mu}(y) \sum_{k=1}^n {x \choose k}  \sum_{a \geq k -p} {p+a \choose a} y^a {p \choose p+a+1-k}.
\end{align*}
Note that the lowest power of $y$ in $\mathcal{V}_{\mu}(y)$ is $\ell(\mu) \geq 1$
and we are only interested in the coefficient of $y^{n-p}$.
So the summand containing $y^a$ with $a \geq n-p$ does not contribute anything.
Thus, the last formula equals
\begin{align*}
\frac{p! (n-p)! }{n!} [y^{n-p}]  \mathcal{V}_{\mu}(y) \sum_{k=1}^n {x \choose k}  \sum_{ a  = k -p}^{n-p-1} {p+a \choose a} y^a {p \choose p+a+1-k}.
\end{align*}
Next, we observe that 
$$
{p \choose p+a+1-k}=0
$$
if $k>n$ and $k-p \leq a \leq n-p-1$, and the last quantity equals zero if $a\geq n-p$.
Then, it is safe to write the last quantity as
\begin{align*}
& \frac{p! (n-p)! }{n!} [y^{n-p}]  \mathcal{V}_{\mu}(y) \sum_{k \geq 1} {x \choose k}  \sum_{ a  \geq  k -p} {p+a \choose a} y^a {p \choose p+a+1-k} \\
=& \frac{p! (n-p)! }{n!} [y^{n-p}]  \mathcal{V}_{\mu}(y) \sum_{a \geq 0} {p+a \choose a} y^a  \sum_{k=1}^{a+p} {x \choose k}   {p \choose p+a+1-k} .
\end{align*}
It is not hard to obtain
\begin{align*}
	& \quad \sum_{k=1}^{a+p} {x \choose k}   {p \choose p+a+1-k} 
	= {x+p-1 \choose p+a} +{x+p-2 \choose p+a} +\cdots +{x \choose p+a}.
\end{align*}
Next, it is easy to see that
\begin{align*}
\sum_{a \geq 0} {p+a \choose a} y^a {x+ i \choose p+a}=\sum_{a \geq 0}  y^a {x+i \choose p}{x+i-p \choose a} ={x+i \choose p} (1+y)^{x+i-p}.
\end{align*}
Therefore, we have
\begin{align*}
	P_{[p, n-p], \mu}(x)  =\frac{p! (n-p)! }{n!} [y^{n-p}]  \mathcal{V}_{\mu}(y) \sum_{i=0}^{p-1} {x+i \choose p} (1+y)^{x+i-p}.
\end{align*}
Then, the proof for $p+2 \leq n-p$ follows. 

	Note that if $n-p=p$ or $n-p=p+1$, then the only $\mu$ satisfying $\min(\mu) \geq p+1$ is $\mu=[n]$.
For these two exceptional cases, $P_{[p, n-p], [n]}(x)$ can be equivalently viewed as the genus distribution polynomial
for certain one-face bicolored maps. By comparing our formula eq.~\eqref{eq:poly-gen}
with Stanley's formula~\cite{Stanely-two} for the one-face case, we find that they agree with each other.
This completes the proof of the theorem.
\qed

Some examples of these polynomials are given below:
\begin{align*}
	P_{[3,5],[4^2]}^{\bullet} & = \frac{5!\times 4}{8!} (51 x^2 + 31 x^4 +2 x^6), \\
	P_{[3,6],[4,5]}^{\bullet} & = \frac{6!}{9! \times 2} (204 x  +  649 x^3 + 150 x^5 + 5  x^7), \\
	P_{[4,6],[5^2]}^{\bullet} & = \frac{6!\times 5}{10! \times 6} (3220 x^2 + 2553 x^4 + 270 x^6 + 5 x^8), \\
	P_{[4,7],[5,6]}^{\bullet} & = \frac{7!}{11! \times 3} ( 3948 x + 14990 x^3 + 4559 x^5 + 260 x^7 + 3 x^9), \\
	P_{[3,9],[4^3]}^{\bullet} & = \frac{9! \times 2}{12! \times 15} ( 1530 x + 6139 x^3 + 2088 x^5 + 141 x^7 + 2 x^9).
\end{align*}

As mentioned, the Harer-Zagier formula~\cite{harer-zagier} is the genus distribution
polynomial for one-face bicolored maps with WDD being regular, i.e., every
white vertex is of degree exactly two.
From this point of view, the following corollary provides
analogues of the Harer-Zagier formula, i.e., the genus
distribution polynomials for two-face bicolored maps with WDD being regular.

\begin{corollary}
	For $p>0$ and $\mu= [k^d]$ where $k>p$ and $dk \geq 2p$, we have
	\begin{align}
	{dk \choose p}	P_{[p, dk-p], [k^d]}^{\bullet}(x) =   \sum_{i=0}^{p-1}
		\sum_{j=0}^d (-1)^{d-j}  {d \choose j}  {x+i \choose p}  {x+i-p+ jk \choose dk-p} .
	\end{align}
\end{corollary}

Let $\mathcal{B}_{p,n}(\mu)$
denote the set of connected bicolored maps of black vertex degree distribution (equivalently, face-type) $[p, n-p]$ and WDD $\mu \vdash n$.
In addition,
denote by $H_n$ the $n$-th harmonic number, i.e., $H_n=1+\frac{1}{2}+\cdots + \frac{1}{n}$.

\begin{proposition}[Average genus] \label{prop:average-genus}
	The expected genus for a randomly chosen bicolored map from $\mathcal{B}_{p,n}(\mu)$ is given by $1/2$ times
	\begin{align*}
		&n-d-H_n- \frac{1}{n+1}\bigg\{ 1+\frac{(-1)^{n-d}}{n}+ \frac{(-1)^p}{{n \choose p-1}} +\frac{(-1)^{n-d-p}}{{n \choose p}} \bigg\}\\  
		& \qquad \qquad + \sum_{S \subset \{\mu_1,\ldots, \mu_d\}} \frac{(-1)^{d-|S|+n-p-\Sigma_S} }{n {n-p-1 \choose  \Sigma_S } {n-1 \choose p}} ,
	\end{align*}
	where $\subset$ means proper subset and $\Sigma_S$ stands for the sum of the elements contained in $S$.
\end{proposition}

\proof
Recall the fact that $\xi_{n,m}(\mathcal{C}_{[p, n-p]}, \mathcal{C}_{\mu})$
equals the number of bicolored maps of black vertex degree distribution $[p, n-p]$ and WDD $\mu$ and $m$
faces.
It is also not hard to see that
$$
|\mathcal{B}_{p,n}(\mu)|= |\mathcal{C}_{[p, n-p]}| \cdot |\mathcal{C}_{\mu}| =\sum_{m>0} \xi_{n,m}(\mathcal{C}_{[p, n-p]}, \mathcal{C}_{\mu}).
$$
We next consider the expected number of faces $\mathbb{E}(m)$ which is obviously given by
$$
\mathbb{E}(m)=\sum_m m \frac{\xi_{n,m}(\mathcal{C}_{[p, n-p]}, \mathcal{C}_{\mu})}{|\mathcal{C}_{[p, n-p]}| \cdot |\mathcal{C}_{\mu}|}=\frac{\textrm{d}}{\textrm{d} x}P_{[p, n-p],\mu}^{\bullet}(x) \bigg|_{x=1} .
$$
We write 
\begin{align*}
P_{[p, n-p],\mu}^{\bullet}(x)= \frac{1}{n!} \sum_{i=0}^{p-1} (x+i)_p \sum_{S \subseteq \{\mu_1,\ldots, \mu_d\}} (-1)^{d-|S|}(x+i-p+ \Sigma_S)_{n-p} .
\end{align*}
The righthand side of the above formula is the sum of $2^d p$ polynomials
which can be classified into three classes: those with the highest power of the factor $x-1$ being respectively
$0, 1,2$.
Obviously, the evaluation at $x=1$ of those with a factor $(x-1)^2$ will contribute zero to $\mathbb{E}(m)$.
There is only one polynomial of which $x-1$ is not a factor, i.e., for $i=p-1$ and $S=\{\mu_1, \ldots, \mu_d\}$.
The contribution of this case is
$$
\frac{1}{n!} p! (n)_{n-p} \sum_{k=1}^n \frac{1}{k}= H_n.
$$
The ones with the highest power of the factor $x-1$ being $1$ are collected in the following:
\begin{align*}
	& \frac{1}{n!} \sum_{i=0}^{p-2} (x+i)_p \big\{(x+i-p+n)_{n-p}+(-1)^d (x+i-p)_{n-p} \big \} \\
	& \qquad + \frac{1}{n!} (x+p-1)_p \sum_{S \subset \{\mu_1,\ldots, \mu_d\}} (-1)^{d-|S|}(x-1+ \Sigma_S)_{n-p}.
\end{align*}
The contribution of the last case is
\begin{align*}
	&	\frac{1}{n!} \sum_{i=0}^{p-2}  \{(1+i-p+n)!  (-1)^{p-i-2} (p-i-2)!+(-1)^d (i+1)! (-1)^{n-i-2}  (n-i-2)!  \} \\
	& \qquad + \frac{p!}{n!} \sum_{S \subset \{\mu_1,\ldots, \mu_d\}} (-1)^{d-|S|}( \Sigma_S)! (-1)^{n-p-1-\Sigma_S} (n-p-1-\Sigma_S)! \, .
\end{align*}
Finally, applying the following well-known identity involving the reciprocal of binomial coefficients
\begin{align}
\sum_{k=0}^p (-1)^k {n \choose k}^{-1}=\frac{n+1}{n+2} \bigg(1+ (-1)^p {n+1 \choose p+1}^{-1}\bigg),
\end{align}
we simplify the last quantity into
$$
\frac{1}{n+1}\bigg\{ 1+\frac{(-1)^{n-d}}{n}+ \frac{(-1)^p}{{n \choose p-1}} +\frac{(-1)^{n-d-p}}{{n \choose p}} \bigg\} 
- \sum_{S \subset \{\mu_1,\ldots, \mu_d\}} \frac{(-1)^{d-|S|+n-p-\Sigma_S} }{n {n-p-1 \choose  \Sigma_S } {n-1 \choose p}} .
$$
Due to the Euler characteristic formula, we have
\begin{align*}
	2 \mathbb{E}(g)= n-d -\mathbb{E}(m),
\end{align*}
and the proof follows from summarizing the above computation for $\mathbb{E}(m)$. \qed

We remark that combining our computational technicality and the generating function expression for $\xi_{n,m}$ in~\cite{louf-phd}, it seems possible
to obtain the genus distribution polynomial $P_{[p, n-p], \mu}(x)$.

\section{Zeros and log-concavity}\label{sec4}

In this section, we discuss the zeros and properties of the coefficients of the connected genus distribution polynomials.
{
For $n-p \geq p>0$ and $\mu=(\mu_1, \mu_2, \ldots, \mu_d) \vdash n$, let 
\begin{align*}
	H_{\mu}(x) = [y^{n-p}] (1+y)^{x-1} \mathcal{V}_{\mu} (y), \quad G_{\mu}(x) = {(x+p-1)_p} H_{\mu}(x).
\end{align*}
The following result of Stanley~\cite{Stanely-two} will be useful to us to analyze the zeros of $H_{\mu}(x)$.

\begin{theorem}[Stanley~\cite{Stanely-two}]\label{thm:stan-zero}
	Let $\textbf{E}$ be the operator $\textbf{E}: f(x) \rightarrow f(x-1)$.
	Suppose $g(t)$ is a complex polynomial with all roots on the unit circle,
	and the multiplicity of $1$ as a root is $m\geq 0$.
	Let $Q(x)= g(\textbf{E}) (x+n-1)_n$.
	If $g(t)$ has degree exactly $d \geq n-1$, then $Q(x)$ is a polynomial of degree $n-m$ for which
	every zero has real part $\frac{d-n+1}{2}$.
\end{theorem}

As an application, we obtain the following proposition.
\begin{proposition}
Every zero of $H_{\mu}(x)$ has real part $-\frac{p-1}{2}$.
\end{proposition}
\proof
We first have
\begin{align*}
	H_{\mu}(x)&=[y^{n-p}] \,  (1+y)^{x-1}  \prod_{i=1}^d \Big\{(1+y)^{\mu_{i}}-1\Big\}  \\
	&=\sum_{1\leq j_1<\cdots < j_i \leq d \atop i \geq 0} {x-1+\mu_{j_1}+\cdots + \mu_{j_i} \choose n-p} (-1)^{d-i}   \\
	&=\sum_{1\leq j_1<\cdots < j_i \leq d \atop i \geq 0} {x-1+n-\mu_{j_1}-\cdots - \mu_{j_i} \choose n-p} (-1)^{i}.
\end{align*}
Denote by $\widetilde{H}(y)$ the polynomial resulted from replacing $x+p$ with $y$ in $H_{\mu}(x)$.
Then, in terms of the backward shift operator $\textbf{E}$, it is not difficult to see
$$
\widetilde{H}(y) = \prod_{i=1}^d (-\textbf{E}^{\mu_i} +1) {y+n-p-1 \choose n-p}.
$$
According to Theorem~\ref{thm:stan-zero}, we conclude that every zero of $\widetilde{H}(y)$
has real part $\frac{n-(n-p-1)}{2}=\frac{p+1}{2}$.
Consequently, all zeros of $H(x)$ have real part $\frac{p+1}{2}-p=-\frac{p-1}{2}$,
completing the proof.
\qed

Next, the following lemma is crucial.

\begin{lemma}\label{lem:zero-base}
	Suppose $z$ is a complex number with $|z|=1$. Then, the real roots of the polynomial $z G_{\mu}(x) + G_{\mu} (x-1)$ include $0, -1, \ldots, 2-p$ and {\color{blue}possibly $\frac{2-p}{2}$}, and every root of $ z G_{\mu}(x) + G_{\mu} (x-1)$ with non-zero imaginary part has real part $-\frac{p-2}{2}$.
\end{lemma}

\proof Since every zero of $H_{\mu}(x)$ has real part $-\frac{p-1}{2}$,
we may assume
$$
G_{\mu}(x)= (x+p-1)(x+p-2)\cdots x \prod_j (x+\frac{p-1}{2}-i b_j),
$$
where $i^2=-1$. 
Suppose $x=-\frac{p-2}{2} + a+ i b$ with $b \neq 0$ is a root of $ z G_{\mu}(x) + G_{\mu} (x-1)$.
Then, $z G_{\mu}(x) + G_{\mu} (x-1)=0$ implies $|G_{\mu}(x) | = |G_{\mu} (x-1)|$.
That is, the following modulus relation holds:
\begin{align*}
& |p-1-\frac{p-2}{2} + a+ i b|\cdot |p-2-\frac{p-2}{2} + a+ i b| \cdots |-\frac{p-2}{2} + a+ i b| \cdot \prod_j |a+\frac{1}{2} +i (b-b_j)|\\
=&|p-2-\frac{p-2}{2} + a+ i b|\cdot |p-3-\frac{p-2}{2} + a+ i b| \cdots |-\frac{p}{2} + a+ i b| \cdot \prod_j |a-\frac{1}{2} +i (b-b_j)|.
\end{align*}
Since $b \neq 0$, the above equality is equivalent to 
\begin{align*}
	& |p-1-\frac{p-2}{2} + a+ i b| \cdot \prod_j |a+\frac{1}{2} +i (b-b_j)|\\
	=&|-\frac{p}{2} + a+ i b| \cdot \prod_j |a-\frac{1}{2} +i (b-b_j)|.
\end{align*}
If $a>0$,
we obviously have $|\frac{p}{2}+a+ib|> |- \frac{p}{2}+a+ib|$ and 
$$
|a+ \frac{1}{2} +i (b-b_j)| > |a - \frac{1}{2}  +i (b-b_j)|
$$
for all $j$. That is, the lefthand side of the last equation is strictly larger than the righthand side.
If $a<0$, we have strict inequality in the opposite direction. Therefore, $a=0$ is necessary,
and thus every root with non-zero imaginary part has real part $-\frac{p-2}{2}$.

It is easy to verify that $0, -1, \ldots, 2-p$ {\color{blue} (being of an empty sequence if $p=1$)} are the roots of $z G_{\mu}(x) + G_{\mu} (x-1)$.
Next, we will show that no other real roots exist {\color{blue}except $\frac{2-p}{2}$}.
Suppose $c \notin \{0, -1, \ldots, 2-p, \frac{2-p}{2}\}$ is a real root.
Then, the following modulus relation holds:
\begin{align*}
	& |p-1+c|\cdot |p-2+c| \cdots |c| \cdot \prod_j |c+\frac{p-1}{2}-i b_j|\\
	=& |p-1+c-1|\cdot |p-2+c-1| \cdots |c-1| \cdot \prod_j |c-1+\frac{p-1}{2}-i b_j|
\end{align*}
which after cancellation becomes
\begin{align*}
	& |p-1+c|\cdot \prod_j |c+\frac{p-1}{2}-i b_j|
	=  |c-1| \cdot \prod_j |c-1+\frac{p-1}{2}-i b_j|.
\end{align*}
Next, we distinguish the following cases:
\begin{itemize}
	\item If $c>1$, then the lefthand side is clearly larger than the righthand side;
	\item If $c=1$, then the lefthand side is larger than zero but the righthand side is zero;
	\item If {\color{blue}$\frac{2-p}{2} < c <1$}, then we can show the lefthand side is larger than the righthand side as every left term is larger than the corresponding right term;
	\item {\color{blue}If $c < \frac{2-p}{2}$, then the lefthand side is smaller than the righthand side as we can show that every left term is smaller than the corresponding right term.}
\end{itemize}
Therefore, no such $c$ exists and the proof is complete. \qed

Lemma~\ref{lem:zero-base} is a kind of generalization of Postnikov and Stanley~\cite[Lemma 9.13]{pos-stan}
where the underlying polynomial is not allowed to have real roots.

\begin{theorem}\label{thm:zero-root}
	Suppose $0 < p \leq n-p$. Then, for $\mu=[p, n-p]$ and any $\mu \vdash n$ with $\min(\mu) \geq p+1$,
	every zero of the genus distribution polynomial $P_{[p, n-p],\mu}^{\bullet}(x)$ has a real part $0$.
\end{theorem}
\proof  We will show in the next section that Theorem~\ref{thm:main-1} holds for $\mu=[p, n-p]$ as well.
As such,
the genus distribution polynomial for the consider cases here can be written as
\begin{align*}
P_{[p, n-p],\mu}^{\bullet}(x) =&(\textbf{E}^{p-1} + \textbf{E}^{p-2} +\cdots +\textbf{E}^0) \frac{(n-p)!}{n!} G_{\mu}(x)\\
=& \frac{(n-p)!}{n!} \prod_{k=1}^{p-1} (\textbf{E} +z_k)  G_{\mu}(x),
\end{align*}
for some $z_k$ with $|z_k|=1$.
Iterating the argument in Lemma~\ref{lem:zero-base}, after each step, we reduce the number of real roots by one
and increase the real part by $1/2$.
Hence, eventually the polynomial has a real root $0$
and every root with possible non-zero imaginary part has real part $-\frac{p-1}{2} + (p-1) \frac{1}{2} =0$.
This completes the proof of the theorem. \qed

}

A sequence of real numbers $s_0, s_1, s_2, \ldots$ is said to have the log-concavity property if $s_i^2 \geq s_{i-1} s_{i+1}$
for all $i>0$, and is said to be unimodal if there exists $m\geq 0$ such that $s_0 \leq \cdots \leq s_m\geq s_{m+1} \geq \cdots$.
It is well known that if a polynomial has only real zeros then the sequence of its coefficients is log-concave,
and log-concavity implies unimodality.
By noticing that $P^{\bullet}_{[p,n-p], \mu}(x)$
either has only terms of odd degrees or only even degrees via a parity argument of permutations,
the following corollary immediately follows from Theorem~\ref{thm:zero-root}.

\begin{corollary}[Log-concavity and unimodality]\label{cor:log-concavity}
	Suppose $0<p \leq n-p$.
	If $\mu=(\mu_1, \mu_2, \ldots, \mu_d) \vdash n$ where $\mu_i \geq p+1$, or $\mu=[p, n-p]$,
	then the sequence 
	$$
	B^0_{p,n}(\mu), B^1_{p,n}(\mu),B^2_{p,n}(\mu), \ldots 
	$$
	 is log-concave and unimodal.
\end{corollary}

Corollary~\ref{cor:log-concavity} is closely related to the conjecture of Gross, Tucker and Robinson~\cite{GRT} which essentially says the numbers of maps of different genera (in increasing order) determined by any fixed connected graph give a log-concave sequence.
The three-decade conjecture~\cite{GRT} has been confirmed for a number of special graph families, but is still open in general (except for a disproof under evaluation recently announced by Mohar~\cite{mohar24}).
In the conjecture, the underlying graph of the maps is fixed, while in our setting, the black vertex degree
distribution and WDD of the underlying bicolored graphs of the bicolored maps under consideration are fixed.

After showing that $P_{[n], \mu}(x)$ has only imaginary zeros for any $\mu \vdash n$, Stanley~\cite{Stanely-two} studied if it holds for $P_{\eta, \mu}(x)$ in general but came up with a counterexample: the following polynomial 
$$
P_{(332), (332)}(x)= \frac{1}{1120} (660 x^2 +424 x^4 +35 x^6 +x^8)
$$
has roots with non-zero real parts.
Stanley wants to know (private communication) if his counterexample
is the ``smallest'' one that has roots with non-zero real parts.
We notice that there are many counterexamples, e.g., $P_{(44),(44)}(x)$ and $P_{(45),(45)}(x)$.
However, their connected counterparts always have only imaginary zeros.
For example, 
$$
P^{\bullet}_{(332),(332)}(x) = \frac{1}{1120} (618 x^2+387 x^4 +3 x^6)
$$
has only imaginary zeros.

In view of Theorem~\ref{thm:zero-root}, Corollary~\ref{cor:log-concavity} and lots of data, we propose the following conjectures.

\begin{conjecture}
	For any $\eta, \mu \vdash n$, the polynomial $P^{\bullet}_{\eta, \mu}(x)$ has only
	imaginary zeros.
\end{conjecture}

\begin{conjecture}
For any $\eta, \mu \vdash n$, the sequence of the non-zero coefficients of the polynomial $P^{\bullet}_{\eta, \mu}(x)$ is 
unimodal and log-concave.
\end{conjecture}

\section*{Acknowledgements}
The authors thank the anonymous referees for comments which improved the presentation of the work.

\vskip 10pt

\noindent {\bf Declarations of interest}: none.

\noindent {\bf Funding}: none.

\appendix
\section{Computation of the $W$-numbers}
In this section, we present the details for obtaining
eq.~\eqref{eq:w-number}.
In the light of Lemma~\ref{lem:chi-pnp} and Lemma~\ref{lem:beta}, we can break the quantity $W_{n, r}$ into the following form:
\begin{align*}
	{	W_{n, r}} &=  \sum_{\lambda=[1^j, n-j] \atop 0 \leq j \leq p-1} \frac{\mathfrak{c}_{\lambda,r} }{	 f^{\lambda} }  \chi^{\lambda}([p, n-p]) \chi^{\lambda}(\mu) 
	+ \sum_{\lambda=[1^j, n-j], \atop n-p \leq j \leq n-1} \frac{\mathfrak{c}_{\lambda,r} }{	 f^{\lambda} }  \chi^{\lambda}([p, n-p]) \chi^{\lambda}(\mu) \\
	& \quad +\sum_{\lambda=[1^j, 2^k, p-k+1, n-j-k-p-1], \atop 0\leq k \leq p-1, \, 0 \leq j \leq n-2p-2} \frac{\mathfrak{c}_{\lambda,r} }{	 f^{\lambda} }  \chi^{\lambda}([p, n-p]) \chi^{\lambda}(\mu).
\end{align*}
We compute the third sum first.
\begin{equation}
	\begin{split}
		&  \quad  A_3
		:	=\sum_{\lambda=[1^j, 2^k, p-k+1,n-j-k-p-1] \atop j,k}\frac{\mathfrak{c}_{\lambda, r}}{f^{\lambda}} \chi^{\lambda}([p, n-p] ) \chi^{\lambda}(\mu )\\
		&=\sum_{ j,k} \Bigg\{ \sum_{h} (-1)^h {r \choose h} {r-h+n-j-k-p-2 \choose n-p} {r-h+p-k-1 \choose p} \frac{(n-p)! p!}{n!} \Bigg \}   \\
		&\qquad \times (-1)^{j+1} \Bigg\{ \sum_{j_l}^{n_{l}-1}\sum_{j_{l+1}}^{n_{l+1}}\cdots\sum_{j_q}^{n_q}{{n_{l}-1}\choose{j_{l}}}{n_{l+1}\choose j_{l+1}}\cdots{{n_q}\choose j_q}\\
		&\qquad  \times \Big\{ (-1)^{j+ \sum_{i \geq l} j_i} \delta_{j+p+1-l,\sum_{i\geq l}i j_i}+  (-1)^{j+1+ \sum_{i \geq l} j_i} \delta_{j+p+1,\sum_{i\geq l} i j_i} \Big\} \Bigg \}.
	\end{split}\nonumber
\end{equation}	
The quantity $A_3$ consists of two parts, and the first part
\begin{equation}
	\begin{split}
		B_1: & =\sum_{ j,k} \Bigg\{ \sum_{h=0}^r (-1)^h {r \choose h} {r-h+n-j-k-p-2 \choose n-p} {r-h+p-k-1 \choose p}  \Bigg \} (-1)^{j+1}  \\
		&\quad \times \Bigg\{ \sum_{j_l}^{n_{l}-1}\sum_{j_{l+1}}^{n_{l+1}}\cdots\sum_{j_q}^{n_q} {{n_{l}-1}\choose{j_{l}}}{n_{l+1}\choose j_{l+1}}\cdots{{n_q}\choose j_q}
		(-1)^{j+ \sum_{i \geq l} j_i } \delta_{j+p+1-l,\sum_{i\geq l}i j_i} \Bigg \} \\
		&= \sum_k  \sum_{h} (-1)^h {r \choose h}  {r-h+p-k-1 \choose p}   
		\Bigg\{  \sum_{j_l}^{n_{l}-1}\sum_{j_{l+1}}^{n_{l+1}}\cdots\sum_{j_q}^{n_q} {{n_{l}-1}\choose{j_{l}}}{n_{l+1}\choose j_{l+1}}\cdots{{n_q}\choose j_q} \\
		& \quad \times (-1)^{1+ \sum_{i \geq l} j_i} {r-h+n-(\sum_{i\geq l}i j_i +l -p-1)-k-p-2 \choose n-p} \Bigg \} - U_1,
	\end{split}\nonumber
\end{equation}	
where 
\begin{align*}
	U_1= \sum_{k=0}^{p-1}  \sum_{h} (-1)^h {r \choose h}  {r-h+p-k-1 \choose p}   
	(-1)^{\ell(\mu)} {r-h-k-1 \choose n-p} .
\end{align*}
In the above computation of $B_1$, we are supposed to sum over all $0\leq j \leq n-2p-2$ and $j_i$'s such
that $j=\sum_{i\geq l}i j_i +l -p-1$ (i.e., the Kronecker delta factor).
However, with a careful analysis, we find that the only case that needs to be
excluded is the case $j_l=n_l-1$ and $j_i=n_i$ for $i>l$ whence the subtraction of $U_1$.

Next, we have
\begin{align*}
	& \quad  B_1+U_1= [y^{n-p}] \sum_k  \sum_{h} (-1)^h {r \choose h}  {r-h+p-k-1 \choose p}   
	\\
	& \quad \times  \sum_{j_l}^{n_{l}-1}\sum_{j_{l+1}}^{n_{l+1}}\cdots\sum_{j_q}^{n_q} {{n_{l}-1}\choose{j_{l}}}{n_{l+1}\choose j_{l+1}}\cdots{{n_q}\choose j_q} (-1)^{1+ \sum_{i \geq l} j_i} (1+y)^{r-h+n-(\sum_{i\geq l}i j_i +l )-k-1}  \\
	& = [y^{n-p}]  \sum_{k=0}^{p-1}  \sum_{h} (1+y)^{r-h+n-l-k-1} (-1)^h {r \choose h}  {r-h+p-k-1 \choose p}   \\
	&	\qquad \times \Big\{ - \{ 1- (1+y)^{-l}\}^{n_l -1} \prod_{i> l} \{ 1- (1+y)^{-i}\}^{n_i} \Big\}.
\end{align*}

Analogously, we compute the other part $B_2$ of $A_3$.
\begin{equation}
	\begin{split}
		B_2: & =\sum_{ j,k} \Bigg\{ \sum_{h} (-1)^h {r \choose h} {r-h+n-j-k-p-2 \choose n-p} {r-h+p-k-1 \choose p}  \Bigg \} (-1)^{j+1}  \\
		&\quad \times \Bigg\{ \sum_{j_l}^{n_{l}-1}\sum_{j_{l+1}}^{n_{l+1}}\cdots\sum_{j_q}^{n_q} {{n_{l}-1}\choose{j_{l}}}{n_{l+1}\choose j_{l+1}}\cdots{{n_q}\choose j_q}
		(-1)^{j+1+ \sum_{i \geq l} j_i } \delta_{j+p+1,\sum_{i\geq l}i j_i} \Bigg \} \\
		&= \sum_k  \sum_{h} (-1)^h {r \choose h}  {r-h+p-k-1 \choose p}   
		\Bigg\{  \sum_{j_l}^{n_{l}-1}\sum_{j_{l+1}}^{n_{l+1}}\cdots\sum_{j_q}^{n_q} {{n_{l}-1}\choose{j_{l}}}{n_{l+1}\choose j_{l+1}}\cdots{{n_q}\choose j_q} \\
		& \quad \times (-1)^{\sum_{i \geq l} j_i} {r-h+n-(\sum_{i\geq l}i j_i  -p-1)-k-p-2 \choose n-p} \Bigg \} - U_2,
	\end{split}\nonumber
\end{equation}	
where 
\begin{align*}
	U_2= \sum_k  \sum_{h} (-1)^h {r \choose h}  {r-h+p-k-1 \choose p}   
	{r-h+n-k-1 \choose n-p} 
\end{align*}
is dealing with the case $j_i=0$ for all $i$.
Next, we have
\begin{align*}
	& B_2+U_2= 
	[y^{n-p}]  \sum_k  \sum_{h} (1+y)^{r-h+n-k-1} (-1)^h {r \choose h}  {r-h+p-k-1 \choose p}   \\
	&	\qquad \times \{ 1- (1+y)^{-l}\}^{n_l -1} \prod_{i> l} \{ 1- (1+y)^{-i}\}^{n_i} .
\end{align*}

{
	
	With a careful inspection, we realize that $\frac{p! (n-p)! }{n!} U_1$ equals the second sum while $\frac{p! (n-p)! }{n!} U_2$
	equals the first sum for $W_{n,r}$. Therefore,
	we now have
	\begin{align*}
		W_{n, r}= &\frac{p! (n-p)! }{n!} (	B_1+U_1+B_2+U_2)\\
		= & [y^{n-p}] \frac{p! (n-p)! }{n!} \mathcal{V}_{\mu}(y) \sum_{k=0}^{p-1}  \sum_{h\geq 0} (1+y)^{r-h-k-1} (-1)^h {r \choose h}  {r-h+p-k-1 \choose p}  \\
		= & [y^{n-p}] \frac{p! (n-p)! }{n!} \mathcal{V}_{\mu}(y) \sum_{k=0}^{p-1} \sum_{a \geq 0} y^a {p+a \choose a} \sum_{h \geq 0} (-1)^h {r \choose h} {r-h-k-1+p \choose p+a} .
	\end{align*}
	Since for $a\geq 0$ and $0 \leq k \leq p-1$,
	\begin{align*}
		\sum_{h \geq 0} (-1)^h {r \choose h} {r-h-k-1+p \choose p+a} & = \sum_{h \geq 0} [t^h](1-t)^r  \cdot [t^{r-h}] \frac{t^{k+1+a}}{(1-t)^{p+1+a}} \\
		&= {p-k-1 \choose r-k-1-a},
	\end{align*}
	we next arrive at eq.~\eqref{eq:w-number}:
	\begin{align*}
		W_{n, r} &= [y^{n-p}]  \frac{p! (n-p)! }{n!} \mathcal{V}_{\mu}(y) \sum_{k=0}^{p-1} \sum_{a \geq 0} y^a {p+a \choose a} {p-k-1 \choose r-k-1-a} \\
		& = [y^{n-p}] \frac{p! (n-p)! }{n!} \mathcal{V}_{\mu}(y) \sum_{a \geq r-p} y^a {p+a \choose a} {p \choose p+a -r +1}.
	\end{align*}
	The last simplification follows from the facts that ${p-k-1 \choose p+a-r} =0$ if $p+a-r <0$ and that
	for $ q_1 \geq 0$ and $ q_2 \geq 0$
	$$
	\sum_{i=0}^{q_2} {i \choose q_1} = {q_2 +1 \choose q_1 +1}.
	$$

\end{document}